\documentclass{amsart} 
 
\usepackage{amssymb}
\usepackage{amsfonts}
\usepackage{graphicx}
\usepackage{multicol,multirow}
\usepackage{mathrsfs}

\usepackage{xcolor}
\usepackage{rotating}
\usepackage{appendix}
\usepackage[numbers]{natbib}
\usepackage{natbib}
\usepackage{amsmath,amssymb,amsthm,enumitem}
\usepackage{tabularx}

\newtheorem{theorem}{Theorem}[section]

\newtheorem{lemma}[theorem]{Lemma}
\newtheorem{corollary}[theorem]{Corollary}

\theoremstyle{definition}

\newtheorem{definition}[theorem]{Definition}

\def\cqfd{
{\hfill
\kern 6pt\penalty 500
\raise -1pt\hbox{\vrule\vbox to 5pt{\hrule width 4pt
\vfill\hrule}\vrule}}
\break}
\title{Trinomials with high differential uniformity}

\author{Yves Aubry, Fabien Herbaut and Ali Issa}

\newcommand{\Addresses}{{ 
  \bigskip
  \footnotesize

  \noindent (Yves~Aubry) \par
  \noindent \textsc{Institut de Math\'ematiques de Toulon - IMATH, Universit\'e de Toulon, France}\par\nopagebreak
  \noindent \textsc{Institut de Math\'ematiques de Marseille - I2M, Aix-Marseille Universit\'e, UMR 7373 CNRS, France }\par\nopagebreak 
  \noindent \textit{E-mail address}: \texttt{yves.aubry@univ-tln.fr, yves.aubry@univ-amu.fr}

  \medskip

  \noindent (Fabien~Herbaut) \par
  \noindent  \textsc{INSPE Nice-Toulon,  Universit\'e C\^ote d'Azur, France}\par\nopagebreak
  \noindent \textsc{Institut de Math\'ematiques de Toulon - IMATH, Universit\'e de Toulon, France}\par\nopagebreak
  \noindent \textit{E-mail address}: \texttt{fabien.herbaut@univ-cotedazur.fr}
  
   \medskip
   
  \noindent (Ali~Issa) \par
  \noindent
  \noindent \textsc{Institut de Math\'ematiques de Marseille - I2M, Aix-Marseille Universit\'e, UMR 7373 CNRS, France}\par\nopagebreak
  \noindent \textit{E-mail address}: \texttt{ali.issa@univ-amu.fr}

}}

\begin{document}

\begin{abstract}
Comparisons of arithmetic and geometric monodromy groups
coupled with the Chebotarev density theorem
enable to
obtain families of trinomials
defined over finite fields of even characteristic
with high differential uniformity
when the base field is large enough.
\end{abstract}
\keywords{Polynomials over finite fields, differential uniformity, Chebotarev theorem, monodromy groups}
\subjclass[2020 Mathematics Subject Classification]{Primary 11T06, 11R58 - Secondary 14G50}

\thanks{This work is partially supported by the French Agence Nationale de la Recherche through the SWAP project under Contract ANR-21-CE39-0012.}

\maketitle
\section{Introduction}

 The theory of polynomials over finite fields is as interesting in its own right as it is for the applications to which it leads.
 The study of the monodromy groups 
  perfectly illustrates this dual interest. 
 Indeed, the
 comparison of the arithmetic and geometric monodromy groups
  enlightens the understanding of polynomials
 as well as it provides  contributions to different fields.
Among these is the determination of the differential uniformity \textit{for most polynomials}.
  
  Recall that the differential uniformity 
  $\delta_{{\mathbb F}_{q}}(f)$, or simply 
  $\delta(f)$, of a polynomial $f\in {\mathbb F}_{q}[x]$  is defined 
 as the maximum number of solutions of the equation $f(x+\alpha)-f(x)=\beta$ in ${\mathbb F}_{q}$
  where $\alpha$ and $\beta$ run over ${\mathbb F}_{q}$ and $\alpha$ is nonzero. 
A first determination of the value of $\delta(f)$ 
     for a generic polynomial $f$ 
has been obtained by Voloch in \cite{V} 
 where he used 
 tools from number theory, namely the Chebotarev density theorem.
 This theorem coupled with the comparison of the monodromy groups enabled him
 to prove
 that for large values of $n$,
 most polynomials of ${\mathbb F}_{2^n}[x]$ of degree $m\equiv 0$ or $3 \pmod 4$ 
 have a differential uniformity equal to $m-1$ or $m-2$.
 
Moreover, an infinite set $\mathcal M$ of odd integer has been introduced in \cite{AHV} 
with the following property: 
if $m\in{\mathcal M}$ is such that $m\equiv 7\pmod 8$,
then for $n$ sufficiently large,
all degree $m$ polynomials $f\in{\mathbb F}_{2^n}[x]$
satisfy $\delta(f)=m-1$.
For the even degree case, a similar result is obtained  in \cite{AHI}
 for a specific family of degrees $m=2^{r}(2^{\ell}+1)$
when $r\geq2$, $\ell \geq 1$ and $\gcd(r,\ell) \leq 2$.

The methods and the sets of degrees we have handled to ensure that some arithmetic and geometric
monodromy groups coincide are quite different depending on the parity of the degrees.
Nevertheless, we manage here to deduce a result from the odd case to the even one.
Indeed, we are able
to transfer
 the property of high differential uniformity
to some trinomials of 
 degree $m$ when $m-1 \in \mathcal{M}$ and $m$ is divisible by $4$.
To be more precise,
the comparison of the monodromy groups
requires
a characterization of Morse polynomials given in an Appendix of Geyer in \cite{JardenRazon}.
This characterization involves the property for a polynomial to have distinct critical values,
and a key point of our work 
(which is developed in subsection \ref{subsection:distinct_critical_values})
is the study of the algebraic set of polynomials of fixed degree which fail to have distinct critical values.

 Finally, our main result has the interesting corollary of supporting the exceptional almost perfect nonlinear conjecture.
 We contextualize and we explain this contribution in the last section.

 
\section{Main result}

First recall that a polynomial $g$ with coefficients in a field $k$ is said to have distinct critical values
if for any $\tau, \eta$ in the algebraic closure $\overline{k}$
the equalities  $g'(\tau)=g'(\eta)=0$ and $g(\tau)=g(\eta)$ imply $\tau=\eta$. 

From now on we will denote by $D_{\alpha}f(x):=f(x+\alpha)-f(x)$ the derivative of $f$ along $\alpha$.
As a consequence of the action of the involution $ x \mapsto x + \alpha$ 
on the set of the roots of $D_{\alpha}f$ one can 
associate to any polynomial $f \in {\mathbb F}_{2^n}[x]$ of degree $m \geq 7$
a unique polynomial $L_{\alpha} f$ of degree less than or equal to $(m-1)/2$ 
such that $L_{\alpha}f (x(x+ \alpha))= D_{\alpha}f(x)$ 
(see Proposition 2.3 of \cite{AHV} and also Proposition 2.1 of \cite{AHI} for more details).

The set $\mathcal M$ is introduced in \cite{AHV} as the set of odd integer $m$ such that
$L_{\alpha} (x^m)$ has distinct critical values.
Proposition 3.11 in \cite{AHV} explains that this
assumption does not depend on the choice of $\alpha$
and
leads to
 the following equivalent definition.

\begin{definition}\label{definition:M}
We define $\mathcal{M}$ as the set of odd positive integers  $m$ such that 
$L_{\alpha}(x^m)$ has distinct critical values (for any nonzero value of $\alpha$) 
or equivalently such that 
\begin{equation}\label{eq:definition_M}
\footnotesize{
\forall  \zeta_1, \zeta_2
\in \overline{\mathbb{F}}_2 \setminus \{ 1 \}, \
\zeta_1^{m-1}  = \zeta_2^{m-1} = 
\left( \frac{1+\zeta_1}{1+\zeta_2}\right)^{m-1} = 1 \Longrightarrow
 \zeta_1=\zeta_2 \textrm{ or } \zeta_1=\zeta_2^{-1}.
 }
\end{equation}
 \end{definition}

It follows immediately from this definition
that if $m$ is odd, 
then $m \in \mathcal{M}$ if and only if $2(m-1) +1 \in \mathcal{M}$,
or if and only if $2^k(m-1) +1 \in \mathcal{M}$ for any nonnegative integer $k$.
And even if $m$ is even,
if $m$ satisfies 
Condition (\ref{eq:definition_M}) 
in Definition (\ref{definition:M})
then
$2^k(m-1) +1 \in \mathcal{M}$ for any $k \geq 1$.

We can now formulate the main result of this paper.
 \begin{theorem}\label{theorem:main}
 Let $m\geq 8$
   be an integer  such that $m\equiv{0}\pmod{4}$ and $m-1\in \mathcal{M}$. 
 For $n$ sufficiently large, 
 if $f(x)=a_0x^m+a_1x^{m-1}+a_{2}x^{m-2}\in {\mathbb F}_{2^n}[x]$
is a polynomial of degree $m$ such that $a_1\neq 0$
 then $\delta_{\mathbb{F}_{2^n}}(f)$ is maximal, 
 that is $\delta_{\mathbb{F}_{2^n}}(f)=m-2.$
 \end{theorem}

To be concrete we conclude this section by 
providing
in the following table
examples of degrees $m$ for which Theorem \ref{theorem:main} applies. \\
 \\
{\small 
\begin{tabularx}{11cm}{|c|p{4cm}|X|}
\hline
    & \textbf{Ex. of degrees $m$ for which Th. \ref{theorem:main} applies} & \textbf{Comments} \\
    \hline
    \hline
     1 & $m= $8 or 12, 20, 24, 28, 36, 40, 48, 52, 56, 60, 68, 76, 80, 84, 88, 96, 108, 112, 116, 120, 124, 132, 136, 140, 144, 160, 164, 168, 176, 192, 196, 200  &
     Degrees $m \leq 200$ for which Th. \ref{theorem:main} applies.
    
      \\
    \hline
    2 & $m=2^k+4$ 
    for $k \geq 2$ & Point (ii) of Proposition 5.2 in \cite{AHV}.\\
\hline
3 & $m=2\ell^{k}+2$  for $k\geq 0$ and \newline 
  $\ell \in \{$3, 5, 11, 13, 17, 19, 23, 29, 37, 41, 43, 47, 53, 59, 61, 67, 71, 79, 83, 97,
  101, 103, 107, 109, 113, 131, 137, 139, 149, 151, 157, 163, 167, 173, 179, 181, 191, 193, 197, 199,$\ldots\}$
& Point (iii) of Proposition 5.2 in \cite{AHV}. \newline \newline 
Holds for any odd prime $\ell$ such that: \newline
- $2^{\ell -1} \not \equiv 1 \pmod {\ell^2}$ 
and 
\newline 
- $m':=\ell+1$ satisfies Condition  (\ref{eq:definition_M}).\\
\hline
\end{tabularx}
}
\ \\

The first list of examples comes from Example 3.16 in \cite{AHV}.
It arises from a computer-assisted checking of 
Condition 
(\ref{eq:definition_M})
which involves an enumeration of the $(m-1)$th roots of unity.

The second family of degrees $m=2^k+4$ is derived from
Point (ii) of Proposition 5.2 in \cite{AHV} where we take $s=2$.

The third family can be deduced from Point (iii) of Proposition 5.2 in \cite{AHV}
where we  take $s=1$. The odd prime $\ell$
has to fulfilll
$2^{\ell -1} \not \equiv 1 \pmod {\ell^2}$ 
while the integer
$m':=\ell+1$ must satisfy Condition (\ref{eq:definition_M}).
The given list of such integers $\ell < 200$ is obtained again 
with the help of a computer algebra system
(example 3.21 in \cite{AHV}).

\section{Proof of the main result}

\subsection{Distinct critical values}\label{subsection:distinct_critical_values}
This subsection aims to control the number of $\alpha$ such that
 $L_{\alpha}f$ fails to have distinct critical values when $f$ is a trinomial
 of the form
 $a_0x^m+a_1x^{m-1}+a_2 x^{m-2}$.
We proceed in two steps.
First we treat the case  when $f$ is a binomial $a_0x^m+a_1x^{m-1}$.
Second we will relate the case of binomials
 to the case of trinomials.

\begin{lemma}\label{lemma:valeurs_critiques_pour_binomes}
    Let $m \geq 8$  be an integer such that
    $m \equiv{0}\pmod{4}$ and  $m-1\in\mathcal{M}$. 
    We set $d=(m-2)/2$.
    For all
     binomials $f(x)=a_0x^m+a_1x^{m-1}\in {\mathbb F}_{2^n}[x]$ such that $a_1\neq 0$, the critical values of $L_{\alpha}f$ are  distinct except for at most
     $(6d+4)\binom{(d-1)/2}{2}$ values of $\alpha\in{\mathbb F}_{2^n}^{\ast}.$
\end{lemma}

\noindent \textit{Proof: }
The appendix of Geyer in \cite{JardenRazon} describes the locus of the degree $d$ polynomials
$\displaystyle g=\sum_{k=0}^{d}b_{d-k} x^{k} \in \mathbb{F}_q[x]$
which fail to have distinct critical values as
the closed set defined by 
 \begin{equation}\label{equation_pi}
 \Pi_d(g):= \prod_{i \neq j} \left( g(\tau_i) - g(\tau_j) \right) 
 \end{equation}
where the $\tau_i$ are the  (double) roots of $g'$.
To be more precise  $\Pi_d(g)$ is a polynomial when $g$ is monic, or else
is an element of $\mathbb{F}_2 [b_0,\ldots,b_d][1/b_0]$.
We point out that
as a consequence of Proposition 2.1 in \cite{AHV}
the polynomial $L_{\alpha} (a_0 x^m + a_1 x^{m-1})$
has degree exactly $d=(m-2)/2$ provided that $a_1 \neq 0$ and even if $a_0=0$.
So when $a_1 \neq 0$ we know that $L_{\alpha} (a_0 x^m + a_1 x^{m-1})$
has distinct critical values if and only if $\Pi_{d} (L_{\alpha} (a_0 x^m + a_1 x^{m-1}))$
is nonzero.

We
set
$e:=\binom{(d-1)/2}{2}$, that is
 the number of ways to choose two different roots of $g'$.
By Lemma 2.8 in \cite{AHI} we know that
$b_{0}^{de} \Pi_{d} (L_{\alpha}f)$ 
is an 
homogeneous polynomial of degree $(6d+4)e$ 
if we consider that  $a_i$ has weight $i$ whereas 
$\alpha$ has weight $1$.
We also know that each term in $b_{0}^{de} \Pi_{d} (L_{\alpha}f)$ 
contains a product of $(d+2)e$ coefficients $a_i$.  
In the case where $f(x)=a_0 x^m + a_1 x^{m-1}$ 
these homogeneity conditions
impose strong constraints on
$b_{0}^{de} \Pi_{d} (L_{\alpha}f)$ which will necessarily take the  form  
 \begin{equation}\label{equation:Pi}
b_{0}^{de} \Pi_{d} (L_{\alpha}f)
=\sum_{i=0}^{(d+2)e} c_i a_0^{(d+2)e-i} a_1^i  \alpha^{(6d+4)e-i}
\end{equation}
where the coefficients $c_i$'s belong to $\mathbb{F}_{2}$.
If we consider $b_{0}^{de} \Pi_{d} (L_{\alpha}f)$ in the ring  $\mathbb{F}_{2}[a_0,a_1][\alpha]$,
the lowest degree in $\alpha$ is possibly $(5d+2)e$ which would correspond
to the term $a_1^{(d+2)e}  \alpha^{(5d+2)}$.

To determine if this monomial does appear in (\ref{equation:Pi})
it is sufficient to evaluate in $a_0=0$ and $a_1=1$.
By definition of $\Pi_{d}$,
the issue comes down to determining  whether
the critical values of $L_{\alpha} (x^{m-1})$ are distinct, which is
 the case because we have supposed that $m-1 \in \mathcal{M}$.

As a consequence, 
for any choice of the $a_i$'s 
in $\mathbb{F}_{2^n}$ such that
$a_1 \neq 0$ 
the polynomial $b_{0}^{de} \Pi_{d} (L_{\alpha}f) \in \mathbb{F}_{2^n}[\alpha]$ 
is nonzero.
As its degree is bounded by 
$(6d+4)e$,
it admits at most $(6d+4)e$ roots which amounts to saying that
there are at most $(6d+4)e$ values of $\alpha$ such that 
$L_{\alpha}(a_0x^m+a_1x^{m-1})$ does not have distinct critical values. \cqfd

The task is now to relate the case of trinomials to the case of binomials.

\begin{lemma}\label{lemma:binomes_vs_trinomes}
Let $m \geq 8$ be an integer
such that $m\equiv{0}\pmod{4}$ and $a_0, a_1, a_2 \in {\mathbb F}_{2^n}$ such that $a_0 \neq 0$ and $a_1\neq0$.
Consider the two polynomials $f(x)=a_0x^m+a_1x^{m-1}+a_2x^{m-2}$ and $h(x)=a_0x^m+a_1x^{m-1}$.
The critical values of $L_{\alpha}f$ are distinct if and only if the critical values of $L_{\alpha}h$ are.
\end{lemma}
\noindent \textit{Proof: }
First we recall that we can reformulate the requirements for $L_{\alpha}f$ to have distinct critical values
 the following way: $f$ shall satisfy
\[
\left \{  
\begin{array}{llcc}
 \textbf{C1} :  \ (D_{\alpha}f)'(\tau)=(D_{\alpha}f)'(\eta)=0 &  & \\
 &  \Longrightarrow  & \tau=\eta \textrm{ or }  \tau=\eta + \alpha.  \\
 \textbf{C2 } : \ D_{\alpha}f(\tau)=D_{\alpha}f(\eta) & & 
\end{array}
\right.
\]
 Indeed if we set $T_{\alpha}(x)=x(x+\alpha)$ one can write $(L_{\alpha}f) \circ T_{\alpha}=D_{\alpha}f$
  and then $(D_{\alpha}f)'=\alpha(L_{\alpha}f)'\circ T_{\alpha}$.
 The result follows from the obvious fact that $T_{\alpha}(\tau)=T_{\alpha}(\eta)$ if and only if $\tau\in\{\eta,\eta+\alpha\}$, as quoted in Lemma 3.7 of \cite{AHV}.
 
We will now prove that in our context 
$f$ satisfies $\textbf{C1}$ and $\textbf{C2}$
if and only if 
$h$ does.
Indeed, for both $f$ and $h$ the condition $\textbf{C1}$ reads
$$a_1(\tau+\alpha)^{m-2}+a_1\tau^{m-2}=a_1(\eta+\alpha)^{m-2}+a_1\eta^{m-2}=0$$
which can be simplified by the nonzero coefficient $a_1$. 
So, when condition $\textbf{C1}$ is satisfied,
the condition $\textbf{C2}$ for $f$ 
which expresses

{\footnotesize 
\[
\begin{array}{rl}
& a_0(\tau+\alpha)^m+a_0\tau^m+a_1(\tau+\alpha)^{m-1}+a_1\tau^{m-1}+a_2(\tau+\alpha)^{m-2}+a_2\tau^{m-2} \\
=& a_0(\eta+\alpha)^m+a_0\eta^m+a_1(\eta+\alpha)^{m-1}+a_1\eta^{m-1}+a_2(\eta+\alpha)^{m-2}+a_2\eta^{m-2}
\end{array}
\]
}
is equivalent to 
{\footnotesize 
$$ 
a_0(\tau+\alpha)^m+a_0\tau^m+a_1(\tau+\alpha)^{m-1}+a_1\tau^{m-1} 
= a_0(\eta+\alpha)^m+a_0\eta^m+a_1(\eta+\alpha)^{m-1}+a_1\eta^{m-1}
$$
}
that is the condition $\textbf{C2}$ for $h$.  It concludes the proof. \cqfd

 \subsection{Application of the Chebotarev density theorem}
Suppose that $f$ satisfies the hypotheses of Theorem \ref{theorem:main}.
The choice of the degree $m\equiv 0 \pmod 4$ and the hypothesis $a_1 \neq 0$ 
imply by Lemma 2.5 in \cite{AHV} that $L_{\alpha}f$ has odd degree $d=(m-2)/2$, 
which is prime to the characteristic of the base field.
Lemma \ref{lemma:valeurs_critiques_pour_binomes} and Lemma \ref{lemma:binomes_vs_trinomes}
ensure that $L_{\alpha}f$ has distinct critical values except for at most 
$(6d+4)\binom{(d-1)/2}{2}$ values
 of $\alpha$. 
 By Proposition 2.5 of \cite{AHI}, the critical points of $L_{\alpha}f$ are nondegenerate (i.e. the derivative $(L_{\alpha}f)'$ and the second
Hasse-Schmidt derivative $(L_{\alpha}f)^{[2]}$ have no common roots)
except for at most $(m-1)(m-4)$ values of $\alpha$ in  ${\overline{\mathbb F}}_2$. 
Frow now on we suppose that
$n$ is sufficiently large, so we can choose
$\alpha$ such that the
three conditions above are satisfied.
As a consequence 
of an analogue of the Hilbert theorem in even characteristic given in the Appendix of Geyer in \cite{JardenRazon},
the geometric monodromy group of $L_{\alpha}f$ is the full symmetric group.
Hence, the splitting field $F$ of $L_{\alpha}f(x)-t$ over ${\mathbb F}_{2^n}(t)$ with $t$ transcendental over ${\mathbb F}_{2^n}$ is a geometric extension of ${\mathbb F}_{2^n}(t)$
 (i.e. there is no constant field extension).

Then 
we consider the splitting field $\Omega$ of the polynomial $D_{\alpha}f(x)-t$ over the field
${\mathbb F}_{2^n}(t)$
and we write $\displaystyle L_{\alpha}f(x)=\sum_{k=0}^db_{d-k}x^k$. 
Proposition 4.6 of \cite{AHV} ensures that if $L_{\alpha}f$ is Morse and if the equation $x^2+\alpha x=b_1/b_0$ has a solution in ${\mathbb F}_{2^n}$
then the extension $\Omega/F$ is 
also geometric.
But Proposition 2.4 of \cite{AHI} states that the  number of $\alpha\in{\mathbb F}_{2^n}^{\ast}$ such that the trace (from ${\mathbb F}_{2^n}$ to ${\mathbb F}_{2}$) of $\frac{b_1}{b_0\alpha^2}$ is equal to zero (i.e. such that  the equation $x^2+\alpha x=b_1/b_0$ has a solution in ${\mathbb F}_{2^n}$) is at least $\frac{1}{2}(2^{n}-2^{n/2+1}-1)$.
We conclude that for $n$ sufficiently large  there exists $\alpha\in{\mathbb F}_{2^n}^{\ast}$ such that the extension $\Omega/{\mathbb F}_{2^n}(t)$ is a geometric Galois extension.

We now use the Chebotarev density theorem
to obtain, once again  for $n$ sufficiently large
depending only on the degree $m$
the existence of a place of degree 1 of ${\mathbb F}_{2^n}(t)$ which totally splits in $\Omega$,
or in other words the existence of $\beta \in {\mathbb F}_{2^n}$ such that the equation 
$f(x+\alpha)-f(x)=\beta$ admits $m-2$ distinct roots.
For this purpose we employ Inequality (7) in \cite{AHI}.

Finally we have proved that $\delta_{{\mathbb F}_{2^n}}(f)=m-2$ for $n$ sufficiently large.


\section{On the exceptional APN conjecture} 
\color{red}  \color{black}

Polynomials $f$ of ${\mathbb F}_{2^n}[x]$ such that $\delta(f)=2$ are called almost perfect nonlinear (APN) and have numerous applications in various fields
(see 
\cite{BN} for a survey).
Such polynomials which are also APN over infinitely many extensions of 
${\mathbb F}_{2^n}$ are called {\sl exceptional} APN
and also receive special attention 
(see for instance \cite{H-McG}, \cite{BS}
and \cite{Delgado} for a survey). 
One conjecture proposed in \cite{A-McG-R} and still open
is whether the only exceptional APN polynomials are
the polynomials
$x^{2^k+1}$ and $x^{2^{2k}-2^k+1}$ for $k \geq 1$,
up to the CCZ equivalence,  
a relation whose definition (\cite{CCZ}) 
is expressed in terms of 
affine permutations of the graphs.

Concretely it is rather difficult to determine whether a polynomial
is APN (or exceptional APN) and
in the two last decades 
many works
have been dedicated to this question.
Quite naturally
the progress achieved 
have often involved lacunary polynomials.
For example results in the direction of the conjecture quoted above
were
 regularly
 obtained 
for polynomials 
$f(x)=x^{2^k+1}+h(x)$ or $f(x)=x^{2^{2k}-2^k+1}+h(x)$
 with extra conditions on $h$ which in particular involve the degree.
This serie of results culminates
with
the recent works \cite{AJD} 
and
\cite{AJD_Kasami}.
Also, binomials of the form 
$ x^{2^s+1}+w x^{2^{ik}+2^{mk+s}}$ with specific conditions on $n$ and $w \in \mathbb{F}_{2^n}^{\ast}$ are
shown to be APN in \cite{BCL}.
In another direction, 
since the introduction in \cite{EKP} of a first APN binomial  $x^3+ux^{36} \in \mathbb{F}_{2^{10}}[x]$
which is not CCZ equivalent to a monomial,
such results have been obtained for trinomials (\cite{BCMG}) and quadrinomials  (\cite{BHK}).
See also \cite{FOR} and \cite{Aub} for examples of treatment for a specific degree. \\

In the case of polynomials of even degree, Theorem \ref{theorem:main}  has the following corollary,
which contributes to the 
exceptional almost perfect nonlinear conjecture.

\begin{corollary}
Let $m\geq 8$
   be an integer  such that $m\equiv{0}\pmod{4}$ and $m-1\in \mathcal{M}$. 
 Polynomials
 $f(x)=a_0x^m+a_1x^{m-1}+a_{2}x^{m-2}\in {\mathbb F}_{2^n}[x]$
 of degree $m$ such that $a_1\neq 0$
 are not exceptional APN.
\end{corollary}

\begin{proof}
 For $n$ sufficiently large, 
 if $f(x)=a_0x^m+a_1x^{m-1}+a_{2}x^{m-2}\in {\mathbb F}_{2^n}[x]$
is a polynomial of degree $m$ such that $a_1\neq 0$
 then Theorem \ref{theorem:main} gives that $\delta_{\mathbb{F}_{2^n}}(f)$ is maximal, 
 that is $\delta_{\mathbb{F}_{2^n}}(f)=m-2.$
 In particular
 such polynomials are not exceptional APN.
\end{proof}

\bibliographystyle{plain}
\bibliography{biblio}
\Addresses
\end{document}